\documentclass[12pt]{amsart}
\usepackage{amssymb}
\usepackage{verbatim}

\begin{document}

\newtheorem{theorem}{Theorem}
\newtheorem{lemma}[theorem]{Lemma}
\newtheorem{corollary}[theorem]{Corollary}
\newtheorem{proposition}[theorem]{Proposition}

\newtheorem{maintheorem}{Main Theorem}
\def\themaintheorem{\unskip}

\theoremstyle{definition}
\newtheorem{remark}{Remark}
\def\theremark{\unskip}
\newtheorem{claim}[remark]{Claim}
\newtheorem{definition}{Definition}
\def\thedefinition{\unskip}
\newtheorem{problem}{Problem}[section]

\numberwithin{equation}{section}

\def\Re{\operatorname{Re\,} }
\def\Im{\operatorname{Im\,} }
\def\distance{\operatorname{distance\,} }
\def\domain{\operatorname{Domain\,} }
\def\e{\varepsilon}
\def\eps{\varepsilon}
\def\p{\partial}

\def\reals{ {{\mathbb R}} }
\def\complex{ {{\mathbb C}} }
\def\torus{ {{\mathbb T}} }
\def\naturals{ {{\mathbb N}} }
\def\integers{ {{\mathbb Z}} }
\def\complex{{\mathbb C}}

\def\scriptd{ {\mathcal D} }
\def\scriptc{ {\mathcal C} }
\def\scripte{ {\mathcal E} }

\def\scriptb{ {\mathcal B} }
\def\scriptt{{\mathcal T}}
\def\scriptf{{\mathcal F}}
\def\scriptg{{\mathcal G}}
\def\scriptv{{\mathcal V}}
\def\scriptl{{\mathcal L}}
\def\scriptn{{\mathcal N}}
\def\scriptm{{\mathcal M}}
\def\lbarlambda{\overline{{\mathcal L}}_\lambda}
\def\dbarlambda{\overline{D}_\lambda}
\def\dbarlambdastar{\overline{D}_\lambda^*}
\def\dbarb{\bar\partial_b}
\def\dbar{\bar\partial}
\def\op{\operatorname{op}}
\def\lt{L^2}
\newcommand{\norm}[1]{ \|  #1 \| }
\newcommand{\Norm}[1]{ \Big\|  #1 \Big\| }
\newcommand{\set}[1]{ \left\{ #1 \right\} }

\author{Michael Christ}
\address{
        Michael Christ\\
        Department of Mathematics\\
        University of California \\
        Berkeley, CA 94720-3840, USA}
\email{mchrist@math.berkeley.edu}
\thanks{The author was supported in part by NSF grant DMS-0901569.}

\date{June 17, 2013}

\title[Upper bounds for Bergman kernels] 
{Upper bounds for Bergman kernels 
\\ associated to  positive Hermitian line bundles \\ with smooth metrics}  

\begin{abstract}
Off-diagonal upper bounds are established away from the diagonal for the Bergman kernels
associated to high powers $L^\lambda$ of holomorphic line bundles $L$ over compact complex manifolds,
asymptotically as the power $\lambda$ tends to infinity. The line bundle
is assumed to be equipped with a  Hermitian metric with positive curvature form, which 
is $C^\infty$ but not necessarily real analytic. The bounds are of the form $\exp(-h(\lambda)\sqrt{\lambda\log\lambda})$
where $h\to\infty$ at a non-universal rate.  This form is best possible.
\end{abstract}

\maketitle

\section{Introduction}

\subsection{The setting}
Let $X$ be a compact complex manifold, without boundary. Let $X$ be equipped with a $C^\infty$ Hermitian metric $g$,
along with the metrics on the bundles $B^{(p,q)}(X)$ of forms of bidegree $(p,q)$ induced by $g$,
and the volume form on $X$ associated to the induced Riemannian metric.
Denote by $\rho(z,z')$ the Riemannian distance from $z\in X$ to $z'\in X$.

Let $L$ be a positive holomorphic line bundle over $X$.  Let $L$ be equipped with a $C^\infty$ Hermitian metric $\phi$ 
whose curvature is positive at every point. $\phi$ is not assumed to be real analytic. 

For each positive integer $\lambda$, let the line bundle $L^\lambda$ be the tensor product of $\lambda$ copies of $L$.
$L^\lambda$ inherits from $\phi$ a Hermitian metric in a natural way; if $v\in L_z$ then
the $\lambda$--fold tensor product $v\otimes v\otimes \cdots\otimes v$ satisfies
$|v\otimes v\otimes \cdots\otimes v| = |v|^\lambda$.

Let $L^2_\lambda= L^2(X,L^\lambda)$ be the Hilbert space of equivalence classes of all square integrable Lebesgue measurable
sections of $L^\lambda$. 
Likewise there are the Hilbert spaces $L^2(X,B^{(0,q)}\otimes L^\lambda)$.
Let $H^2_\lambda$ be the closed subspace of $L^2_\lambda$ consisting of all holomorphic sections.  
The Bergman projection is defined to be the orthogonal projection $B_\lambda$ from $L^2_\lambda$ onto $H^2_\lambda$.
The Bergman kernel $B_\lambda(z,z')$ is the associated distribution-kernel; 
$B_\lambda(z,z')$ is a complex linear endomorphism from the fiber $L^\lambda_{z'}$ to the fiber $L^\lambda_z$.

A great deal is known concerning the nature of these Bergman kernels. 
In particular, detailed asymptotic expansions are known near the diagonal $z=z'$,
that is, when $\rho(z,z')$ is bounded by a constant multiple of $\lambda^{-1/2}$.
See for instance \cite{bermanBS},\cite{catlin},\cite{tian},\cite{zelditch}
as well as the related work \cite{boutetsjostrand} of Boutet de Monvel and Sj\"ostrand on the
Bergman and Szeg\"o kernels associated to domains in $\complex^{n+1}$. 
This paper is concerned with upper bounds when $z,z'$ are far apart, that is,
behavior for large $\lambda$ when $\rho(z,z')$ is bounded below by a positive
quantity independent of $\lambda$. If $\phi$ and $g$ are real analytic, then for large $\lambda$, 
$|B_\lambda(z,z')|\le C_\delta e^{-c_\delta\lambda}$ whenever $\rho(z,z')\ge\delta>0$, 
where $C_\delta<\infty$ and $c_\delta>0$ are independent of $\lambda$.
This is interpreted in the theory of Bleher, Shiffman and Zelditch \cite{BSZ1},\cite{BSZ2},\cite{shiffmanzelditch}
of random zeroes of sections of $L^\lambda$ as an exponentially small upper
bound on the degree of correlation between zeros at distinct points. 

\subsection{Subexponential off-diagonal decay}
It was shown in \cite{christcounter} that this exponential decay fails to hold, in general,
if $\phi$ is merely infinitely differentiable. More quantitatively,
for any function $h$ satisfying $h(t)\to\infty$ as $t\to+\infty$ there exists \cite{christcounter} an example for which 
\begin{equation}
\limsup_{\lambda\to\infty} \sup_{\rho(z,z')\ge\delta} e^{h(\lambda)\sqrt{\lambda\log\lambda}}\,{|B_\lambda(z,z')|}=\infty
\end{equation}
for all $\delta>0$.
In this paper we establish an upper bound which dovetails with these lower bounds.

\begin{theorem}\label{thm:main}
Let $L$ be a positive holomorphic line bundle over a compact complex manifold $X$.
Let there be given a $C^\infty$ positive metric on $L$ with strictly positive curvature form,
and a $C^\infty$ Hermitian metric on $X$.
For any $\delta>0$ there exist $\Lambda<\infty$ and a function $h$ satisfying
$h(\lambda)\to\infty$ as $\lambda\to\infty$
such that for all $z,z'\in X$ satisfying $\rho(z,z')\ge\delta$,
\begin{equation}
|B_\lambda(z,z')|\le e^{-h(\lambda)\sqrt{\lambda\log\lambda}}\ \text{ for all $\lambda\ge\Lambda$.}
\end{equation}
\end{theorem}

The analysis below of $B_\lambda$ is based on its connection with the fundamental solution of
a partial differential operator, $\square_\lambda$.
Denote by $\dbar_\lambda$ the usual Dolbeault operator, mapping sections
of $B^{(0,q)}\otimes L^\lambda$ to sections of $B^{(0,q+1)}\otimes L^\lambda$.
Denote by $\dbar_\lambda^*$ its formal adjoint, with respect to the Hilbert space structures $L^2_\lambda$ defined above.
Define
\begin{equation}
\square_\lambda = 
\begin{cases}
\dbar_\lambda^*\dbar_\lambda + \dbar_\lambda\dbar_\lambda^*\qquad & \text{for $n>1$}
\\
 \dbar_\lambda\dbar_\lambda^* &\text{for $n=1$},
\end{cases}
\end{equation}
acting on sections of $B^{(0,1)}\otimes L^\lambda$.
For each $\lambda$, $\square_\lambda$ is an elliptic second-order linear system of partial differential
operators with $C^\infty$ coefficients. When it is expressed in local coordinates, its coefficients
are $O(\lambda^2)$ in any $C^N$ norm.

Because the metric $\phi$ is positive, there exists a constant $c>0$ such that
for all sufficiently large $\lambda\in\naturals$,
\begin{equation} \label{formlowerbound} \langle \square_\lambda u,u\rangle \ge c\lambda\norm{u}_{L^2}^2 \end{equation}
for all twice continuously differentiable sections $u$ of $B^{(0,1)}\otimes L^\lambda$.
This bound is deduced from a well-known integration by parts calculation \cite{hormander}.
Because of this lower bound and because $\square_\lambda$ is formally self-adjoint and elliptic,
there exists a unique self-adjoint bounded linear operator $G_\lambda$ on 
$L^2(X,B^{(0,1)}\otimes L^\lambda)$ satisfying $\square_\lambda\circ G_\lambda=I$, the identity operator.


The operator $B_\lambda$ is related to $\square_\lambda$ by 
\begin{equation}\label{szegoformula} 
B_\lambda = I-\dbar_\lambda^*\circ G_\lambda\circ\dbar_\lambda.  \end{equation}
Thus the Bergman kernel is expressed in terms of certain derivatives of the distribution-kernel
for the operator $G_\lambda$. We denote this distribution-kernel by $G_\lambda(z,z')$.
Because $G_\lambda(z,z')$ is a solution of $\square_\lambda G_\lambda=0$ with respect to
the variable $z$ and its complex conjugate is a solution of the same equation with respect to $z'$, 
elliptic regularity theory guarantees that $G_\lambda(z,z')$ is a $C^\infty$
function of $(z,z')$ on the complement of the diagonal.

We will show that $G_\lambda(z,z')=O(e^{-h(\lambda)\sqrt{\lambda\log\lambda}})$
for $(z,z')$ at any positive distance from the diagonal.  The corresponding bound
holds for those partial derivatives that express the distribution-kernel for 
$\dbar_\lambda^*\circ G_\lambda\circ\dbar_\lambda$ at $(z,z')$ will be an easy consequence.

\medskip
For real analytic metrics, the Bergman kernel is $O(e^{-c\lambda})$ away from the diagonal.
Combining the result established here with that of \cite{christcounter}, one knows that
for $C^\infty$ metrics, decay can in some instances be essentially as slow as $e^{-h(\lambda)\sqrt{\lambda\log\lambda}}$,
but is never slower.
Zelditch has raised the question of which, or what, behavior is typical, and whether properties of
the metric can be inferred from the off-diagonal decay rate of the associated Bergman kernels. 
This issue will be examined in \cite{christzeldconjecture}.

\subsection{Orientation}
A weaker upper bound $|B_\lambda(z,z')| \le e^{-c\sqrt{\lambda}}$, valid
whenever $\rho(z,z')\ge\delta$, is a simple consequence of \eqref{formlowerbound},
and requires only $C^2$ or even $C^{1,1}$ regularity of $\phi$.
In the context of global analysis on $\complex^1$, this was shown in \cite{christcounter}.
For positive line bundles over complex manifolds, it was noted by Berndtsson \cite{bob}. 
Closely related results are found in works of Delin \cite{delin} and Lindholm \cite{lindholm}.
The novelty in Theorem~\ref{thm:main} is a double improvement of the exponent, from $c\sqrt{\lambda}$
to $h(\lambda)\sqrt{\lambda\log\lambda}$.

To establish the weaker bound, consider any real-valued auxiliary weight
$\psi\in C^2(X)$. For any $\eps>0$ and all sufficiently large $\lambda$,
\begin{equation}\label{formlowerboundtwisted}  
\begin{aligned}
\Re\big(\langle e^{\eps\sqrt{\lambda}\psi}
\square_\lambda & e^{-\eps\sqrt{\lambda}\psi} u,u\rangle\big) 
\\ &\ge \langle \square_\lambda  u,u\rangle 
-C\lambda^{1/2} \eps \norm{\dbar_\lambda u}\cdot\norm{u}
-C\lambda^{1/2} \eps \norm{\dbar_\lambda^* u}\cdot\norm{u}
-C\lambda\eps^2 \norm{u}^2
\\& \ge (c-C\eps)\lambda\norm{u}_{L^2}^2 \end{aligned} \end{equation}
for all sections $u\in C^2(X,B^{(0,1)})$, where $C$ depends on the $C^2$ norm of
$\psi$. This is $\ge \norm{u}^2$ for all sufficiently
large $\lambda$, provided that $\eps$ is chosen to be sufficiently small as a function of
$\norm{\psi}_{C^2}$. 
The inequality \eqref{formlowerboundtwisted} can alternatively be interpreted as a weighted inequality
for the inverse operator $\square_\lambda^{-1}$, with weight $e^{2\eps\sqrt{\lambda}\psi}$.
Whenever $U,U'$ are disjoint sets satisfying $\distance(U,U')\ge\delta>0$,
by choosing $\psi$ so that $\psi\ge 1$ on $U'$ and $\psi\le 0$ on $U$ 
we conclude that $\square_\lambda^{-1}$ maps $L^2(U)$ to $L^2(U')$, where these norms are defined
without reference to the auxiliary weight $\psi$, with operator norm
$O(e^{-c\sqrt{\lambda}})$ where $c>0$ depends on $\delta$. The pointwise bound for $B_\lambda(z,z')$
for $(z,z')\in U\times U'$ is a simple consequence, by a routine elliptic regularity bootstrapping 
argument which will be used below in the main body of the proof.

\section{Variants}

Here are two variants of Theorem~\ref{thm:main}, concerning metrics $\phi$ with more limited regularity.
For simplicity we continue to assume that the underlying Hermitian/Riemannian metric on $X$ itself is $C^\infty$.

\begin{theorem}\label{thm:two}
Let $n\ge 1$, let $X$ be a compact complex manifold, and let $L$ be a postitive holomorphic line bundle over $X$.
Let $\delta>0$. There exist a positive integer $N_0$ and $\eta>0$
such that for every $N\ge N_0$  and every $C^N$ metric on $L$ with positive curvature
there exists $\Lambda$ such that
\begin{equation}
|B_\lambda(z,z')|\le e^{-\eta\sqrt{N\,\lambda\,\log\lambda}}.
\end{equation}
whenever $\rho(z,z')\ge\delta$ and $\lambda\ge\Lambda$.
\end{theorem}

A natural question is whether the indicated dependence on $N$, as $N\to\infty$, is optimal. 
The construction in \cite{christcounter} could be used to obtain a bound in the opposite
direction; we have not reexamined the details to determine whether it shows that
the bound obtained here is optimal.
A proof of Theorem~\ref{thm:two} is implicit in the proof given below of Theorem~\ref{thm:main};
it is simply a matter of tracing the dependence on $N$ of the auxiliary parameter $A$ introduced
in that proof.

The following variant is not proved in this paper, but could be established by augmenting the method used here
with a more precise version of Lemma~\ref{lemma:three}, below. Such a refinement can be established
by arguments closely related to those in \cite{christweighted}.
For $\alpha\in(0,1)$ let $C^{2,\alpha}$ denote the class of all $C^2$ functions whose second order partial
derivatives are all H\"older continuous of order $\alpha$.

\begin{claim}\label{thm:three}
Let $n\ge 1$
and $\alpha\in(0,1)$. 
Let $L$ be a holomorphic line bundle over a compact complex manifold $X$ of complex dimension $n$.
Let there be given a positive metric of regularity class $C^{2,\alpha}$ on $L$, 
and a $C^\infty$ Hermitian metric on $X$.
Then for each $\delta>0$ there exist $\Lambda<\infty$ and $\eta>0$ 
such that 
for all $\lambda\ge\Lambda$
and for any open sets $U,U'\subset X$ satisfying $\rho(U,U')\ge\delta$,
for any section $f\in L^2(X,L^\lambda)$ supported in $U'$,
\begin{equation}
\norm{B_\lambda f}_{L^2(U)} \le e^{-\eta \sqrt{\lambda\,\log\lambda}}\norm{f}_{L^2}.
\end{equation}
\end{claim}
These $L^2$ norms are computed with respect to the weight $\phi$.
The proof uses Taylor expansion of degree $2$, rather than of high degree.  

\section{Unweighted bounds and twisted operators}

It will be convenient to work in an equivalent framework, in a coordinate patch in $X$,
in which $L$ is trivial and norms are defined by integrals without $\lambda$--dependent weights,
but the underlying operators $\dbar_\lambda$, $\square_\lambda$ are twisted.
This framework is more natural for discussion of regularity. 

Let $U$ be a small coordinate patch on $X$, over which $L$ may be
identified with $\complex$. Functions and differential forms may be regarded as scalar--valued. 
For each $q$, $\dbar_\lambda$, mapping sections of $B^{(0,q)}\otimes L^\lambda$ over $U$
to sections of $B^{(0,q+1)}\otimes L^\lambda$ over $U$,
is naturally identified with the standard Cauchy-Riemann operator $\dbar$,
which maps sections of $B^{(0,q)}$ to sections of $B^{(0,q+1)}$.

$\phi\in C^\infty$ is $\reals$-valued, and the positive curvature assumption means precisely
that its complex Hessian matrix 
$\begin{pmatrix} \partial^2\phi/\partial z_j\partial\bar z_k \end{pmatrix}$
is strictly positive definite at each point of $U$.
The $C^\infty$ Hermitian metric $g$ given for $X$ is interpreted as a $C^\infty$
Hermitian metric on $U$, and gives rise to a volume form, expressed as
a measure $\mu$ on $U$, which is a smooth nonvanishing multiple of Lebesgue
measure on $\complex^n$. It also gives rise, for each $q$, to a $C^\infty$ metric
on $B^{(0,q)}$ over $U$.
The $L^2$ norm squared of a section of $B^{(0,q)}$ over $U$, regarded
as a scalar-valued function $f$, is expressed as $\int_U |f(z)|^2 e^{-2\lambda \phi(z)}\,d\mu(z)$,
where $|f(z)|$ is measured according to $g$.

Substituting $fe^{-\lambda \phi}=u$,
the norm squared of $f$ with respect to the weight $\phi$ 
becomes $\norm{f}_{L^2}^2 = \int_U |u(z)|^2\,d\mu(z)$; there is no weight in this integral.
Moreover \begin{equation} e^{-\lambda \phi} \dbar f = e^{-\lambda \phi} \dbar(ue^{\lambda \phi}) 
= \dbar u + \lambda au \end{equation} where $a=\dbar \phi\in C^\infty$.
For each $q$ define
\begin{equation} \bar D_\lambda = e^{-\lambda\phi}\circ\dbar\circ e^{\lambda\phi}
= \bar\partial+\lambda a \cdot.  \end{equation} 
This is a first-order linear partial differential operator with smooth coefficients, but with
a zero-th order term proportional to the large parameter $\lambda$.
The formal adjoint(s) $\bar D_\lambda^*$ are defined with respect to the given metric $g$ and associated volume form.
These data are assumed to be only $C^\infty$, rather than $C^\omega$,
but their potential lack of analyticity is less significant than that of $\phi$ because they
are not multiplied by the large parameter $\lambda$. 

Define
\begin{equation}
\Delta_\lambda = 
\begin{cases} \dbarlambda\dbarlambda^* + \dbarlambda^*\dbarlambda\qquad  &\text{for $n>1$,}
\\
\dbarlambda\dbarlambda^* &\text{for $n=1$,}
\end{cases}
\end{equation}
acting on $(0,1)$ forms over $U$. 
Under these identifications,
\begin{equation*}
\Delta_\lambda = e^{-\lambda\phi}\circ\square_\lambda\circ e^{\lambda\phi}.
\end{equation*}
The function
\begin{equation}
\scriptg_\lambda(z,w) = e^{-\lambda\phi(z)+\lambda\phi(w)}G_\lambda(z,w)
\end{equation}
represents a fundamental solution for $\Delta_\lambda$ with pole at $w$, in the usual sense.
This is a section of the complex endomorphism bundle of $B^{(0,1)}$ over $U\times U$ minus the diagonal;
in this local coordinate system, it is a matrix-valued function.
Its size $|\scriptg_\lambda(z,w)|$ is defined with respect to given smooth metrics 
which do not depend on $\lambda$, so upper bounds with respect to these metrics are uniformly equivalent
to upper bounds with respect to the standard metrics on these bundles.

Theorem~\ref{thm:main} is therefore equivalent to an upper bound for all $(z,w)$ in $U\times U$
minus the diagonal of the form
\begin{equation} |\scriptg_\lambda(z,w)|\le e^{-A\sqrt{\lambda\log\lambda}} 
 \ \text{ for all $\lambda\ge\Lambda(\delta,A)$, whenever $|z-w|\ge\delta$ }
\end{equation}
with corresponding upper bounds for  all first and second--order derivatives of $\scriptg_\lambda$
with respect to $z,w$ in this same region.

\section{A near-diagonal upper bound}

Theorem~\ref{thm:main},
which is concerned with the nature of $G_\lambda$ far from the diagonal, will be derived from a 
description of $G_\lambda$ much nearer the diagonal.
The main point is the manner in which the bounds depend 
on $\lambda,A$; these bounds are completely independent of the exponent $A$, 
provided only that $\lambda$ exceeds a certain threshold, which does depend on $A$.
The reasoning below will require bounds for derivatives of $G_\lambda$, as well as
for $G_\lambda$ itself. These bounds are more naturally expressed in terms
of the twisted kernels $\scriptg_\lambda$ introduced above.
$\nabla$ will denote the gradient in $\complex^n\times\complex^n$, with respect to both coordinates $z,z'$.


\begin{proposition} \label{prop:neardiagonal}
There exist $c_0,A_0\in\reals^+$ such that
for any $A\in[A_0,\infty)$ there exists $\Lambda=\Lambda(A)<\infty$
such that for any $\lambda\ge\Lambda$ and any $z,z'\in U$ satisfying
\begin{equation} 
A_0\lambda^{-1/2}\log\lambda
\le |z-z'|
\le A\lambda^{-1/2}\log\lambda,
\end{equation}
$\scriptg_\lambda(z,z')$ satisfies
\begin{equation}
|\scriptg_\lambda(z,z')| 
+ |\nabla_{z,z'}\scriptg_\lambda(z,z')| 
\le 
e^{-c_0\lambda|z-z'|^2}.
\end{equation}
\end{proposition}

Here $\op$ denotes the  operator norm on the Hilbert space $L^2(X,B^{(0,1)}\otimes L^\lambda)$.


As is well understood, there is a natural scale $\asymp \lambda^{-1/2}$ inherent in this situation.
In the model situation in which $X=\complex^n$ and $\phi(z)\equiv \tfrac12 |z|^2$,
$|\scriptg_\lambda(z,z')|\asymp e^{-c_0 \lambda |z-z'|^2} |z-z'|^{2-2n}$ for $n>1$, with the
power of $|z-z'|$ replaced by $\log(1/|z-z'|)$ for $n=1$. 
Proposition~\ref{prop:neardiagonal} asserts essentially that this model upper bound persists up to a distance
which is greater by a multiplicative factor of $A\sqrt{\log\lambda}$
than the natural scaled distance, for arbitrarily large $A$. 
The lower bound $|z-z'|\ge A_0\lambda^{-1/2}\sqrt{\log\lambda}$
is an inesssential technicality introduced in order to simplify the statement and proof of the lemma;
otherwise the upper bound would have to be modified in order to take the near-diagonal factor $|z-z'|^{2-n}$ into account. 

In the next section we will show how Theorem~\ref{thm:main} is an essentially formal consequence
of Proposition~\ref{prop:neardiagonal}. 
We will then review and establish foundational results, none of which involve significant novelty, 
before finally proving the Proposition.

\section{The near-diagonal bound implies the far-from-diagonal bound}

$\norm{T}_{\text{op}}$ will denote the operator norm of $T$,
as an operator on $L^2(X,B^{(0,1)}\otimes L^\lambda)$.
Recall that $\rho$ denotes the Riemannian distance function on $X^2$.
The following obvious statement is at the heart of the construction.

\begin{lemma} \label{lemma:addsupports}
Let $T_1,T_2$ be bounded linear operators on $L^2(X,B^{(0,q)}\otimes L^\lambda)$.
Let $r_i>0$ and suppose that for $i=1,2$, the distribution-kernel associated to $T_i$ 
is supported in $\set{(z,z')\in X^2: \rho(z,z')\le r_i}$.
Then the distribution-kernel associated to $T_1\circ T_2$
is supported in $\set{(z,z')\in X^2: \rho(z,z')\le r_1+r_2}$.
\end{lemma}

This will be used to prove:
\begin{lemma} \label{lemma:two}
Let $A<\infty$ and $\delta>0$.
There exist $C<\infty$ and $\Lambda<\infty$
such that for every $\lambda\ge\Lambda$ there  exists
a bounded linear map $T$ 
from the space of $L^2$ sections of $B^{(0,1)}\otimes L^\lambda$
to itself with these two properties: 
Firstly, the distribution-kernel for $T$ is supported in $\set{(z,z'): \rho(z,z')\le\delta}$.
Secondly, 
\begin{equation} \norm{T\circ\square_\lambda-I}_{\text{op}} \le e^{ -A \lambda^{1/2} \sqrt{\log\lambda}}.  \end{equation}
\end{lemma}

\begin{proof}
Choose an auxiliary function $\eta\in C^\infty([0,\infty))$ that satisfies
$\eta(x)\equiv 1$ for $x\le\tfrac12$, and $\eta(x)\equiv 0$ for all $x\ge 1$.
Let $A<\infty$. 
Let $P$ be the operator with distribution-kernel
\begin{equation*}
K(z,w) = G_\lambda(z,w)\eta(A^{-2}\lambda(\log\lambda)^{-1}\rho^2(z,w)).
\end{equation*}
Letting $\square_\lambda$ act with respect to the $z$ variable,
and applying Leibniz's rule and the chain rule,
\begin{equation*}
|\square_\lambda (K(z,w)-G_\lambda(z,w))|
\le C \lambda^2 |G_\lambda(z,w)|   + C\lambda^2 |\nabla G_\lambda(z,w)|.
\end{equation*}
$\square_\lambda K(z,w)$ is supported where $\rho(z,w)\asymp A\lambda^{-1/2}(\log\lambda)^{1/2}$.
In this region, according to Proposition~\ref{prop:neardiagonal},
\begin{equation*} |G_\lambda(z,w)| + |\nabla G_\lambda(z,w)|
\le C \lambda^C e^{-c\lambda A^2 \lambda^{-1} \log\lambda}
\le C\lambda^{C-cA^2}.  \end{equation*}
So in all,
\begin{equation*} |\square_\lambda (K(z,w)-G_\lambda(z,w))| \le \lambda^{C-cA^2}
\end{equation*}
for all sufficiently large $\lambda$, uniformly for all pairs $(z,w)$ in $X^2$ minus the diagonal.
Since $\square_\lambda\circ G_\lambda=I$, this is an upper bound for the operator norm of $\square_\lambda\circ P-I$.
Since both $\square_\lambda$ and $P$ are formally self-adjoint, 
the same bound holds for $P\circ\square_\lambda-I$.

Given $\delta>0$, choose $N$ to be the largest integer such that
$N A \lambda^{-1/2}(\log\lambda)^{1/2}\le\delta$.  Thus
\begin{equation*} N\asymp A^{-1}\lambda^{1/2}(\log\lambda)^{-1/2}\delta.  \end{equation*}

Set 
\begin{equation*} E =I-\square_\lambda \circ P \qquad\text{ and } \qquad
T = P\circ \sum_{j=0}^{N-1} E^j \end{equation*}
so that
\begin{equation*} \square_\lambda\circ T = I-E^N.  \end{equation*}
Because the distribution-kernel for $P$ is supported where $\rho(z,w)\le A\lambda^{-1/2}\sqrt{\log\lambda}$,
the distribution-kernel for $T$ is supported where
\begin{equation*} \rho(z,w)\le NA \lambda^{-1/2} \sqrt{\log\lambda} \le\delta, \end{equation*}
according to Lemma~\ref{lemma:addsupports}.

Since $\norm{E}_{\text{op}} = \norm{\square_\lambda\circ P-I}_{\text{op}} \le \lambda^{C-cA^2}$,
\begin{equation*} \norm{E^N}_{\text{op}} \le \lambda^{(C-cA^2)N}
\le \lambda^{(C-cA^2)c(A^{-1}\lambda^{1/2}(\log\lambda)^{-1/2}\delta)}
\le e^{-c'A\lambda^{1/2}\sqrt{\log\lambda}} \end{equation*}
for all sufficiently large $A$.  \end{proof}

\begin{proof}[Proof of Theorem~\ref{thm:main}]
Consider any $z'\ne z''\in X$.
To prove the upper bound for $B_\lambda(z',z'')$,
consider any $L^2$ section $f$ of $B^{(0,1)}\otimes L^\lambda$
that is supported in $B''=B(z'',\tfrac14\rho(z',z''))$ and satisfies $\norm{f}_{L^2}\le 1$.
Choose $T$ as in Lemma~\ref{lemma:two}, with distribution-kernel supported
within distance $\tfrac12\rho(z',z'')$ of the diagonal.
Then in $B'=B(z',\tfrac14\rho(z',z''))$,
\begin{align*}
G_\lambda f &= T\square_\lambda G_\lambda f + O\big(e^{-A\sqrt{\lambda\log\lambda}}\norm{G_\lambda f}\big)
\\&= Tf + O\big(e^{-A\sqrt{\lambda\log\lambda}} \norm{f}\big).
\end{align*}
Since $T$ has distribution-kernel supported within distance $\tfrac12\rho(z,z')$ of the diagonal, $Tf\equiv 0$ in $B'$.
Therefore
\begin{equation*}
G_\lambda f =  O(e^{-A\sqrt{\lambda\log\lambda}}\norm{f}) \text{ in $L^2(B')$ norm.}
\end{equation*}
Thus as an operator from $L^2(B'')$ to $L^2(B')$, $G_\lambda$ has operator norm
$O(e^{-A\sqrt{\lambda\log\lambda}})$. 
Because $G_\lambda(z,w)$ is a solution of elliptic linear partial differential
equations with $C^\infty$ coefficients with respect to both variables $z,w$, and because the coefficients of
those equations are $O(\lambda^2)$ in every $C^N$ norm, it follows from standard bootstrapping arguments that
for any $N$, $G_\lambda\in C^N(B'\times B'')$, with norm $O(e^{-A\sqrt{\lambda\log\lambda}})$. 
Since the Bergman kernel is the distribution-kernel for
$I-\dbar^*_\lambda G_\lambda \dbar_\lambda$, this result with $N=2$
includes the desired upper bound.
\end{proof}

\section{Off-the-shelf upper bounds}
\subsection{Low regularity upper bounds}

Thus far the argument has been purely formal.  We now state two quantitative estimates on which 
the proof of Proposition~\ref{prop:neardiagonal} will rely.  One concerns metrics with nearly minimal 
regularity; the other, real analytic metrics.  The $C^\infty$ case is intermediate between the two.

\begin{lemma} \label{lemma:three}
For each $n\ge 1$ there exists $N<\infty$ with the following property.
Let $L$ be a positive holomorphic line bundle over a compact complex manifold $X$
of dimension $n$, equipped with a Hermitian metric $\phi$ of class $C^N$.
Assume that likewise that $X$ is equipped with a Hermitian metric $g$ of class $C^N$. 
Let $U,\scriptg_\lambda$ be as defined above.
Then there exists $C<\infty$ such that for all sufficiently large positive integers $\lambda$,
\begin{equation}\label{weaklocalupperbounds}
|\scriptg_\lambda(z,z')| + |\nabla \scriptg_\lambda(z,z')|  \le (\lambda+|z-z'|^{-1})^C
\end{equation}
for all $z\ne z'\in U$.
\end{lemma}

\noindent Here $\nabla$ denotes the gradient with respect to both variables $z,z'$.

Considerably sharper upper bounds can be established, but they will not be needed 
in the proof of Theorem~\ref{thm:main}.

\begin{proof}
The fact that integration with respect to $\scriptg_\lambda$ defines a
bounded operator on $L^2(U,B^{(0,1)})$, uniformly in $\lambda$
can be interpreted as a weak {\it a priori} upper bound for $\scriptg_\lambda$, as for instance in \cite{christweighted}.
Routine localization and bootstrapping arguments, exploiting the ellipticity of $\Delta_\lambda$
and the $O(\lambda^2)$ bounds for its coefficients, lead directly to \eqref{weaklocalupperbounds}.
\end{proof}

\subsection{High regularity upper bounds}
We work now in the unweighted twisted framework introduced above.  Let $B\subset \complex^n$ be any 
fixed open ball of positive radius, and let $\tilde B\Subset B$ be any relatively compact subball.

\begin{lemma} \label{lemma:four}
Let the ball $B\subset\complex^n$ be equipped with a $C^\omega$ Hermitian metric $g$.
Let $L$ be any holomorphic line bundle over $B$,
equipped with a positive $C^\omega$ Hermitian metric $\phi$. 
There exist $\Lambda<\infty$ and $c>0$ such that for any $\lambda\ge\Lambda$
and any solution $u$ of $\Delta_\lambda u\equiv 0$ on $B$ 
\begin{equation} |u(z)|\le e^{-c\lambda}\norm{u}_{L^2(B)} \text{ for all } z\in\tilde B.  \end{equation}
\end{lemma}
Moreover, given a family of such metrics $g,\phi$, $c$ may be taken to be independent of $g,\phi$, 
provided that $g,\phi$ are uniformly $C^\omega$ and that the metrics $\phi$ are uniformly positive.

Positivity of $\phi$ means that in local coordinates, 
$\sum_{i,j=1}^n \frac{\partial^2 \phi(z)}{\partial z_i\partial \bar z_j}\zeta_i\bar\zeta_j
\ge a|\zeta|^2$ for all $\zeta\in\complex^n$ and all $z$,
for some $a>0$. We say that a family of metrics $\phi$ is uniformly positive
if $a$ is bounded below by some positive constant uniformly for all elements of the family in question.
Likewise we say that such a family is uniformly $C^\omega$ if there exists $C<\infty$ for which 
\begin{equation} \big|\frac{\partial^\alpha\phi}{\partial(z,\bar z)^\alpha}\big|
\le C^{1+|\alpha|} |\alpha|!\ \text{ uniformly on $B$}\end{equation} 
for every multi-index $\alpha$ and all metrics $\phi$.  The same applies to $g$.

\begin{proof}[Proof of Lemma~\ref{lemma:four}]
This is a consequence of a fundamental result on analytic hypoellipticity
of related subelliptic partial differential equations.
Consider first the case $n>1$. Work in $B\times\reals^1$
with coordinates $(z,t)$, and set $U(z,t) = u(z)e^{i\lambda t}$.
Then \[e^{i\lambda t}\dbarlambda u(z)  = \dbarb U(z,t),\]
where $\dbarb$ is a Cauchy-Riemann operator associated 
to a strictly pseudoconvex CR structure on $B\times\reals$;
$\Delta_\lambda$ is related to the Kohn Laplacian $\square_b$ for
this CR structure by the corresponding equation
\[e^{i\lambda t}\Delta_\lambda u(z) = \square_b U(z,t).\] 

For $n>1$, $\square_b$ is analytic hypoelliptic on $(0,1)$ forms,
for any $C^\omega$, strictly pseudoconvex CR structure.
Proofs of this and/or closely related results can be found in 
\cite{chinni},\cite{grigissjostrand},\cite{sjostrand2},\cite{tartakoff},\cite{treves2}.
Identifying $B\subset\complex^n$ with a ball in $\reals^{2n}$,
we regard $B\times\reals$ as a subset of $\reals^{2n+1}$, hence as
a totally real submanifold of $\complex^{2n+1}$. Any real analytic function
of $(z,t)\in B\times\reals$ thus extends holomorphically to a neighborhood in $\complex^{2n+1}$.

Analytic hypoellipticity of $\square_b$ implies such extension, in a
quantitative sense: there exist a complex neighborhood $\Omega$
of $\tilde B\times [-1,1]$ 
and a constant $C<\infty$ such that any bounded solution $U$ of 
$\square_b U=0$ in $B\times(-2,2)$ extends to a bounded holomorphic
function in $\Omega$, and moreover, 
\[\sup_\Omega|U|\le C\sup_{B\times(-2,2)} |U|.\]

By analytic continuation, any holomorphic extension of $u(z)e^{i\lambda t}$ must take the product form 
$\tilde u(z) e^{i\lambda t}$.  For positive $\lambda$ we then set $t=-i$ to deduce that
\[\sup_{\tilde B}|u| e^{\lambda}\le C\sup_B |u|.\]

An examination of any of the proofs 
\cite{sjostrand2, tartakoff, treves2}
of analytic hypoellipticity of $\square_b$ confirms that these provide {\em uniform} upper bounds,
given uniform upper bounds on the coefficients 
of $\dbarb$ in some fixed coordinate patch, and on the Hermitian metric used to define $\dbarb^*$, 
and given that the hypothesis of strict pseudoconvexity holds in a uniform
way. In our setting, the latter amounts to uniform strict positivity of the metric $\phi$.

The case $n=1$ requires an alternative treatment, because $\square_b =
\dbarb\dbarb^*$ fails to be analytic hypoelliptic for three-dimensional CR manifolds.
Instead, a variant of analytic hypoellipticity holds in two alternative (but
equivalent) forms. One of these\footnote{The other alternative asserts
that $u$ is $C^\omega$, microlocally outside a conic neighborhood of one
of the two ray bundles whose union is the characteristic variety of $\dbarb$.
This implies holomorphic extendibility to an appropriate wedge, and the
above reasoning may then be repeated to gain the factor $\exp(-c\lambda)$.}
asserts that if $\dbar\dbar^*U=0$ then 
\begin{equation}\sup_{\Omega}|\dbar^* U| \le C\sup_{B\times(-2,2)} (|U|+|\dbar^* U|),\end{equation} 
with the same type of uniform dependence of the constant $C$ on the data as for $n>1$.
Together with the reasoning above, this yields the conclusion
\begin{equation} \sup_{\tilde B}|\dbarlambdastar u|
\le e^{-c\lambda}\sup_B (|u|+ |\dbarlambdastar u|).\end{equation}

The bound for $u$ itself now follows from Lemma~\ref{lemma:recoveru} below.
\end{proof}

The justification of the above form of analytic hypoellipticity rests on several facts and results,
combined according to an outline introduced by Kohn \cite{kohn} for the analysis of related questions
concerning (weakly) pseudoconvex three-dimensional CR manifolds.
Denote by $\square=\dbar_b\dbar_b^*$ 
the Kohn Laplacian over a strictly pseudoconvex three (real) dimensional CR manifold $M$.
Assume that $\square u\in C^\omega$ in an open set.

\smallskip
\noindent
(i) The analytic wave front set of $u$ is contained in the characteristic variety of $\square$.
\newline
(ii) This characteristic variety is a real line bundle over $M$, thus a union of two ray bundles. 
\newline
(iii) In a conic neighborhood of one of these two ray bundles, $\dbar_b$ is of principal type
and satisfies (microlocally) the Poisson bracket hypothesis which ensures analytic hypoellipticity
\cite{treves1}, and therefore is microlocally analytic hypoelliptic. 
The microlocal version of this theorem of Treves follows for instance by the techniques in \cite{sjostrand1}.
Consequently since $\dbar_b(\dbar_b^* u)\in C^\omega$, the analytic wave front set of $\dbar_b^* u$
is disjoint from this ray bundle.
\newline
(iv) In a conic neighborhood of the complementary ray bundle, $\square$ has double characteristics and
satisfies the hypotheses of the theorem of Sj\"ostrand \cite{sjostrand2}; see also \cite{grigissjostrand}
where more degenerate operators are analyzed by the same techniques. 
Therefore  the analytic wave front set of $u$, and hence also the analytic wave front set of $\dbar_b^* u$, are
disjoint from this ray bundle.
\newline
(v) If a distribution has empty analytic wave front set, then it is analytic.
\newline
(vi) These steps can be made quantitative, where appropriate, to justify the stated uniformity.

\subsection{Exponential localization for a first-order equation}
\begin{lemma} \label{lemma:recoveru}
Let $n\ge 1$. Let $U,U'$ be open subsets of $X$ with $U\Subset U'$.
There exists $c>0$ such that for all sufficiently large $\lambda\ge 0$,
and all $u\in C^1(U')$,
\begin{equation*}
\norm{u}_{L^2(U)} \le C\norm{D_\lambda^* u}_{L^2(U')} 
+ Ce^{-c\lambda} \norm{u}_{L^2(U')}.
\end{equation*}
\end{lemma}

\begin{proof}
It suffices to show that for each $z_0\in U$, there exists a neighborhood $V$
of $z_0$ such that $\norm{u}_{L^2(V)}$ satisfies the required upper bound.
In a small open set, represent $\bar D_\lambda^*$
as $-e^{\lambda\phi}(\partial +a) e^{-\lambda\phi}$ where $a\in C^\infty$.
In a sufficiently small neighborhood it is possible to solve $\partial\alpha=a$ and thus to write
$\bar D_\lambda^*  = -e^{-\alpha} e^{\lambda\phi}\partial e^{-\lambda\phi}e^{\alpha}$.  
Since multiplication by $e^{\pm\alpha}$ preserves $L^2$ norms up
to a bounded factor, it suffices to prove the inequality with $\alpha\equiv 0$.

It is possible to write, for all $z,w$ in a sufficiently small neighborhood of $z_0$,
\begin{equation*}
\phi(w) = \psi(z,w) + \varphi(z,w)
\end{equation*}
where $\psi,\varphi$ are $C^\infty$ functions, $\varphi(z,w)$ is an antiholomorphic
function of $w$ for each $z$,
and 
\begin{equation} \Re(\psi(z,w))\ge \Re(\psi(z,z))+c|z-w|^2 \end{equation} 
for a certain constant $c>0$.
Such a decomposition is obtained by exploiting the Taylor series of order $2$ for $\phi$ at $z$.
Then for each $z$, when acting on functions of $w$,
\begin{equation*}
\bar D_\lambda^* u(w) = -e^{\lambda\psi(z,w)}\big(\partial e^{-\lambda\psi(z,\cdot)}\big)u(w).
\end{equation*}

Let $\eta\in C^\infty(X)$ be a function supported in a neighborhood
of $z_0$ which is contained in a coordinate patch contained in a relatively compact
subset of $U'$, within which the above expression for $\phi$ is valid; 
and $\eta$ is identically equal to one in a smaller neighborhood. 
Then $\eta u$ can be regarded as a function defined on $\complex^1$.
Let \begin{equation*} v=\bar D_\lambda^* (\eta u) = \eta \bar D_\lambda^* u - u\partial \eta.\end{equation*}
Since 
\begin{equation*} \partial_w e^{-\lambda\psi(z,w)}(\eta u)(w)=-e^{-\lambda\psi(z,w)} v(w) \end{equation*}
is a compactly supported continuous function defined on $\complex^1$,
for each $z$ sufficiently close to $z_0$ one may recover $\eta(z)u(z)=u(z)$ by 
\begin{equation}
u(z) = -c_0 \int_{\complex^1} v(w) (\bar z-\bar w)^{-1} e^{\lambda(\psi(z,z)-\psi(z,w))}\,dm(w)
\end{equation}
where $m$ denotes Lebesgue measure on $\complex^1$ and $c_0$ is a certain constant.
Now
\begin{equation*}
\big|e^{\lambda(\psi(z,z)-\psi(z,w))}\big|
 = e^{\lambda(\Re(\psi(z,z)-\psi(z,w)))}
\le e^{-c\lambda |w-z|^2}.
\end{equation*}
Therefore
\begin{align*}
|u(z)| 
&\le  C \int_{\complex^1}  |z-w|^{-1}|v(w)| e^{-c\lambda|z-w|^2} \,dm(w)
\\&\le  C \int_{\complex^1}  \big(|\eta(w)u(w)| + |u(w)\partial\eta(w)|\big)\,|z-w|^{-1} e^{-c\lambda|z-w|^2} \,dm(w)
\end{align*}
Since $|z-w|$ is bounded below by a positive quantity uniformly for all $z$ in $U$ and $w$ in the support of $\nabla\eta$,
the required bound follows.
\end{proof}

\section{Proof of Proposition~\ref{prop:neardiagonal}}

\subsection{Globalization}
We introduce a variant situation in which 
$X$ is replaced by $\complex^n$ and sections of $B^{(0,1)}\otimes L^\lambda$ over $X$ are replaced
by sections of $B^{(0,1)}(\complex^n)$ over $\complex^n$.
This variant will facilitate $\lambda$--dependent coordinate changes to be made below.

Let $\e>0$ be given. Let $U$ be a relatively compact open subset of a coordinate
patch in $X$.  Fix a holomorphic coordinate system on  that coordinate patch,
and express $\dbarlambda = e^{-\lambda\phi}
\dbar e^{\lambda\phi}$ where $\phi\in C^\infty$ is $\reals$-valued,
and satisfies the positivity hypothesis
\begin{equation} \left(\frac{\p^2\phi}{\p z_i \p \bar z_j}\right)_{i,j} \ge c(\delta_{i,j})_{i,j}\end{equation}
in the sense of Hermitian forms.

Sections of $L^\lambda$ over $U$ are thus identified with $\complex$--valued functions
in such a way that the $\lt$ norm squared, over $U$, of such a section can be expressed 
as $\int_U |f(z)|^2 a(z)\,d\mu(z)$ where $\mu$ is Lebesgue measure on $\complex^n$,
$a\in C^\infty(\complex^n)$ is bounded above in $C^N$ norm for all $N$ by constants 
independent of $\lambda,z'$, and $a(z)$ is positive and bounded below by a positive constant
independent of $\lambda,z,z'$.
Extend $a$ to a strictly positive $C^\infty$ function $\tilde a$ on $\complex^n$,
still with uniform upper and lower bounds.
Likewise extend $g$ to a $C^\infty$ Hermitian metric on $\complex^n$, independent of $\lambda$.
Assign to $(0,k)$ forms $f$ defined on $\complex^n$ the $\lt$ norm squared
$\int_{\complex^n} |f(z)|^2\tilde a(z)\,d\mu(z)$
where $|f(z)|$ is measured using this extension of $g$.

Fix an auxiliary function $\eta\in C^\infty_0(\complex^n)$, supported in $\set{z: |z|<4}$
and satisfying $\eta(z)\equiv 1$ for $|z|\le 2$.
For each $z'$ in a fixed relatively compact subset $U\Subset U'$, make the affine coordinate change
\[B\times U\owns (\zeta,z')\mapsto (z,z')=(z'+\zeta,z')\in U\times U,\] 
where $B$ is the ball of radius $\eps_0$ centered at $0\in\complex^n$.
In these coordinates, $z'$ is the origin, $\zeta=0$.
We will work in the variable $z\in B$, suppressing $z'$ in the notation;
all estimates will be uniform in $z'\in U$, as the proof will show.

Let $Q_2$ be the Taylor polynomial of degree $2$ for $\phi$ at $\zeta=0$.
Define
\begin{equation*}
\tilde \phi(\zeta) = Q_2(\zeta) + \eta(\eps_0^{-1} \zeta)(\phi(\zeta) -Q_2(\zeta)).
\end{equation*}
Consider the modified operator 
$e^{-\lambda\tilde \phi}\dbar_\zeta e^{\lambda\tilde \phi}$,
which agrees with $e^{-\lambda \phi}\dbar_\zeta e^{\lambda \phi}$ for all sufficiently small $\zeta$,
but has the advantage of being defined globally for $\zeta\in\complex^n$.
For sufficiently large $\lambda$,
\begin{equation*} \nabla^2 \tilde\phi(z)-\nabla^2\phi(0) = O(\eps_0)
\end{equation*}
uniformly for all $z\in\complex^n$. 
Therefore it is possible to choose $\eps_0>0$ sufficiently small that for all sufficiently large $\lambda$,
the quadratic form defined by $(\partial^2 \tilde\phi(z)/\partial z_i \partial \bar z_j)_{i,j=1}^n$ is bounded below
by a strictly positive constant, independent of $z$ and $\lambda$.
This holds uniformly in $z'\in U$.  Choose and fix such a value of $\eps_0$.

Consider the associated operator defined for $n>1$ by
\begin{equation*}
\tilde\Delta_\lambda
=
\big(e^{-\lambda\tilde \phi}\dbar e^{\lambda\tilde \phi}\big)
\big(e^{-\lambda\tilde \phi}\dbar e^{\lambda\tilde \phi}\big)^*
+
\big(e^{-\lambda\tilde \phi}\dbar e^{\lambda\tilde \phi}\big)^*
\big(e^{-\lambda\tilde \phi}\dbar e^{\lambda\tilde \phi}\big),
\end{equation*}
and for $n=1$ by
\begin{equation*}
\tilde\Delta_\lambda
=
\big(e^{-\lambda\tilde \phi}\dbar e^{\lambda\tilde \phi}\big)
\big(e^{-\lambda\tilde \phi}\dbar e^{\lambda\tilde \phi}\big)^*,
\end{equation*}
where adjoints are interpreted with respect to the Hilbert space
structure on $\lt(\complex^n)$ introduced above.

For $n>1$,
for all sufficiently large $\lambda$, 
a well-known computation based on integration by parts \cite{hormander} gives
\begin{equation} \label{quadformlowerbound}
\langle \tilde\Delta_\lambda u,u\rangle \ge c\lambda \norm{u}_{L^2}^2
\end{equation}
for all twice continuously differentiable and compactly supported $(0,1)$ forms $u$,
where $c>0$ is a positive constant.

For $n=1$, for all sufficiently large $\lambda$, 
\begin{equation} \big[e^{-\lambda\tilde \phi}\dbar e^{\lambda\tilde \phi},
\big(e^{-\lambda\tilde \phi}\dbar e^{\lambda\tilde \phi} \big)^*\big] \ge c\lambda I,\end{equation} 
in the sense of operators on $L^2(\complex^n)$ with respect to the same Hilbert space structure.
Consequently \eqref{quadformlowerbound} also holds for $n=1$.

Since
$\tilde\Delta_\lambda$ is a formally self-adjoint operator,
it follows that there exists a bounded linear operator $\tilde \scriptg_\lambda$ from $L^2(\complex^n, B^{(0,1)})$ to itself 
such that $\tilde\Delta_\lambda\circ \tilde \scriptg_\lambda$ is the identity operator on $L^2(\complex^n, B^{(0,1)})$,
and the operator norm of $\tilde \scriptg_\lambda$ is $O(\lambda^{-1})$ for all sufficiently large $\lambda$.

This inverse is bounded in $L^2$ operator norm, uniformly for all sufficiently large $\lambda$,
provided that $\eps_0$ is kept fixed.
Lemma~\ref{lemma:three} also applies to this situation, so
the distribution-kernel $\tilde \scriptg_\lambda(z,0)$
for $\tilde \scriptg_\lambda$ with pole at $\zeta=0$ satisfies
\begin{equation}
|\tilde \scriptg_\lambda(z,0)| \le
(\lambda + |z|^{-1})^C
\end{equation}
for all sufficiently large $\lambda$,
and the same holds for all of its partial derivatives.
These bounds are uniform in $\lambda$ provided that $\lambda$ is sufficiently large.

\subsection{Gauge change}

Denote by $p$ the harmonic part of the Taylor polynomial of $\tilde \phi$ of degree $2$ at $w=0$.  That is, expand
\begin{equation*}
\tilde\phi(z) 
= \tilde\phi(0) + \Re\big(\sum_{k=1}^n \alpha_k z_k + \sum_{i,j=1}^n \beta_{i,j} z_iz_j\big)
+ \sum_{i,j=1}^n \gamma_{i,j} z_i\bar z_j + O(|z|^3),
\end{equation*}
and set
\begin{equation*}
p(z) = \tilde\phi(0) + \Re\big(\sum_{k=1}^n \alpha_k z_k + \sum_{i,j=1}^n \beta_{i,j} z_iz_j\big).
\end{equation*}
Denote by
$\tilde p$ the real-valued harmonic conjugate of $p$, normalized to vanish at $0$.
Then
$[\dbar, e^{\lambda(p+i\tilde p)}] = \dbar(p+i\tilde p)\equiv 0$
and consequently
\begin{equation}
e^{-\lambda\tilde \phi} \dbar e^{\lambda\tilde \phi} 
= e^{i\lambda\tilde p} e^{-\lambda(\tilde \phi-p)} \dbar
e^{\lambda(\tilde \phi-p)} e^{-i\lambda\tilde p}.
\end{equation}
Likewise
\begin{align*}
\big(e^{-\lambda\tilde \phi} \dbar e^{\lambda\tilde \phi} \big)^*
= \big( e^{i\lambda\tilde p} e^{-\lambda(\tilde \phi-p)} \dbar
e^{\lambda(\tilde \phi-p)} e^{-i\lambda\tilde p} \big)^*
= e^{i\lambda\tilde p} \big( e^{-\lambda(\tilde \phi-p)} \dbar
e^{\lambda(\tilde \phi-p)}\big)^* e^{-i\lambda\tilde p}
\end{align*}
and consequently
\begin{multline*}
e^{-\lambda\tilde \phi} \dbar e^{\lambda\tilde \phi} 
\big(e^{-\lambda\tilde \phi} \dbar e^{\lambda\tilde \phi} \big)^*
+
\big(e^{-\lambda\tilde \phi} \dbar e^{\lambda\tilde \phi} \big)^*
e^{-\lambda\tilde \phi} \dbar e^{\lambda\tilde \phi} 
\\
= e^{i\lambda\tilde p} 
\Big(e^{-\lambda(\tilde \phi-p)} \dbar e^{\lambda(\tilde \phi-p)}
\big( e^{-\lambda(\tilde \phi-p)} \dbar e^{\lambda(\tilde \phi-p)}\big)^* 
\\
+
\big( e^{-\lambda(\tilde \phi-p)} \dbar e^{\lambda(\tilde \phi-p)}\big)^* 
e^{-\lambda(\tilde \phi-p)} \dbar e^{\lambda(\tilde \phi-p)}\Big)
e^{-i\lambda\tilde p}.
\end{multline*}

Hence upon replacement of $\tilde\phi$ by $\tilde\phi-p$ in the definition of $\tilde\square_\lambda$, 
a unitarily equivalent operator on $L^2(\complex^n,B^{(0,1)})$ is obtained.
Moreover, the absolute value of the distribution-kernel for the inverse
of this unitarily equivalent operator is identically equal to $|\tilde G_\lambda|$.

In deriving upper bounds for $|G_\lambda(z,w)|$,
where $G_\lambda$ is the distribution-kernel for $\square_\lambda^{-1}$ on $X$, 
we may therefore assume without loss of generality that the harmonic part of the
Taylor polynomial of degree $2$ for $\phi$ at $w$ vanishes identically.
Likewise, because $\dbar_\lambda$ and $\dbar_\lambda^*$ have been conjugated
by the unitary multiplicative factor $e^{i\tilde p}$, the same assumption can be made
when deriving upper bounds for $|\dbar_\lambda G_\lambda(z,w)|$ and $|\dbar_\lambda^* G_\lambda(z,w)|$.

\subsection{Taylor expansion and dilation}\label{subsect:Taylor}

Let $\tilde\phi$ be as above, and suppose, as we may achieve through a gauge change, that 
the harmonic portion of the Taylor polynomial of degree $2$ for $\tilde\phi$ at $0$
vanishes identically, while the complex Hessian matrix of $\tilde\phi$ is bounded below
by a strictly positive constant, and all partial derivatives of $\tilde\phi$ are bounded
above, uniformly in $\lambda$.

Let $N$ be a large positive integer, independent of $\lambda$, to be chosen below.
Define $P_N$ to be the Taylor polynomial of degree $N$ for $\tilde\varphi$, at $\zeta=0$.
For any $r>0$ satisfying $\lambda^{-1/2}\le r\le \lambda^{-1/4}$
define
\begin{equation} \psi(z) = r^{-2} P_N(rz) + r^{-2}(1-\eta(z))(P_2(rz)-P_N(rz)).  \end{equation}
For all sufficiently large $\lambda$,
the complex Hessian of $\psi$ evaluated at an arbitrary point $z\in\complex^n$,
equals the complex Hessian of $\tilde\phi$ evaluated at $0$, plus $O(r)= O(\lambda^{-1/4})$.

Moreover on $\set{z: |z|<3}$, where $1-\eta\equiv 0$,
$\psi$ is real analytic, {\em uniformly in $\lambda$ and in  $N$} 
provided that $\lambda\ge\Lambda(N)$ where $\Lambda(N)$ is some
appropriately large quantity depending only on $N$ and the data $X,L,\phi,g$.
This uniformity, which is crucial to our analysis,
is a consequence of the normalizations $\tilde\phi(0)=0$, $\nabla\phi(0)=0$ achieved by
subtracting the degree one Taylor polynomial\footnote{
Subtraction of the harmonic second degree terms is natural, but inessential here.}
of $\tilde\phi$; indeed, for $z$ in any bounded set and $N\ge 2$, $P_N(rz)=P_2(rz) + O(r^{3}|z|)$ 
so that \[r^{-2}P_N(rz)=P_2(z) +O_{M,N}(r)\]  in any $C^M$ norm on any bounded set. 
Once $M,N$ are chosen, the term $O_{M,N}(r)$ becomes arbitrarily small as $\lambda$
becomes arbitrarily large.

Define a globalized locally analytic approximation $g^\dagger$ to the Hermitian metric $g$ by
\begin{equation*} g^\dagger(z)  = P_N(rz) + (1-\eta(z))(g(0)-P_N(rz)) \end{equation*}
where now $P_N$ is the Taylor polynomial of degree $N$ for $g$ at $0$, in the natural sense.

Define
\begin{equation} \kappa = r^2\lambda\end{equation}
and
\begin{equation} \bar \scriptd = e^{-\kappa\psi}\bar\partial e^{\kappa\psi}, \end{equation}
that is, $\bar\scriptd u = e^{-\kappa\psi} \bar\partial(e^{\kappa\psi}u)$, for $(0,q)$ forms $u$ defined on $\complex^n$.
Define $\bar\scriptd^*$ to be the adjoint of $\bar\scriptd$ with respect to the 
Hilbert space structures on $L^2$ sections of $B^{(0,q)}(\complex^n)$ specified by $g^\dagger(z)$.
Define
\begin{equation*} \square^\dagger = 
\begin{cases} \ \bar\scriptd\bar\scriptd^* + \bar\scriptd^*\bar\scriptd \ \ &\text{for $n>1$,}
\\ \ \bar\scriptd\bar\scriptd^* &\text{for $n=1$.}
\end{cases} \end{equation*}

These are differential operators.  On the region $|z|<4$,
$\square^\dagger$ is related to $\square_\lambda$ as follows: If $u(z)=v(rz)$ then
\begin{equation} \label{dilationtwining}
\square^\dagger u(z) = r^2\square_\lambda v(rz) 
+ O(\lambda^{-cN})\,
O(v,\dbar_\lambda v,\dbar_\lambda^* v, \dbar_\lambda(bv), \dbar_\lambda^*(bv))
\end{equation}
where the error term denoted
$O(v,\dbar_\lambda v,\dbar_\lambda^* v, \dbar_\lambda(bv), \dbar_\lambda^*(bv))$
is a linear combination of $v$, $\dbar_\lambda(v)$, $\dbar_\lambda^*(v)$,
$\dbar_\lambda(bv)$ and $\dbar_\lambda^*(bv)$ where all coefficients are bounded uniformly
in $\lambda,z$, and 
$bv$ denotes either the wedge product or the interior product of $v$ with a real analytic $(0,1)$ form $b$. 
Moreover, in this region, these forms $b$ are uniformly analytic as $\lambda\to\infty$.


Applying \eqref{dilationtwining} with \[u(z) = \scriptg_\lambda(rz,0),\]
using the upper bounds $|\scriptg_\lambda(z,0)| \le C\lambda^C$ and 
$|\dbar_\lambda \scriptg_\lambda(z,0)| + |\dbar_\lambda^* \scriptg_\lambda(z,0)| \le C\lambda^C$
for $|z|\ge \lambda^{-1/2}$,
and using the assumption $\lambda^{-1/2}\le r\le \lambda^{-1/4}$, we conclude that  
\begin{equation*} |\square^\dagger u(z)| \le  \lambda^{-cN} \end{equation*}
for $\tfrac12\le|z|\le 2$,
where $c>0$ is independent of $\lambda,z$ and of $N$, provided that $\lambda\ge \Lambda(N)$.

Provided that $\kappa=r^2\lambda$ is sufficiently large, the standard integration by parts calculation 
together with the uniform lower bound for the complex Hessian of $\psi$ give the lower bound
\begin{equation}\label{kappahessianbound} \langle \square^\dagger u,u\rangle  \ge c\kappa \norm{u}_{L^2}^2 \end{equation}
for all $C^2$ forms $u$ of bidegree $(0,1)$ with compact support. 
The effect of the localization and rescaling has been to replace $\lambda$ by $\kappa$.

\subsection{Conclusion of Proof of Proposition~\ref{prop:neardiagonal}}
Let $N$ be a large positive integer.
Suppose that $\lambda$ is large, that $\lambda^{-1/2}\le r\le \lambda^{-1/4}$, and that $\kappa=r^2\lambda$ is large.
Consider $u(z) = \scriptg_\lambda(rz,0)$, defined as above using Taylor polynomials of order $N$.
In the annular region $\tfrac12<|z|<2$,
$|u|\le\lambda^C$ and $|\square^\dagger u| \le\lambda^{-cN}$, provided that $\lambda\ge\Lambda(N)$.

Let $\tilde\eta$ be a $C^\infty$ function which is identically equal to $1$ in $\set{z: \tfrac13\le|z|\le 3}$
and supported in $\set{z: \tfrac14<|z|<4}$.
Provided that $\kappa$ is sufficiently large, the global lower bound \eqref{kappahessianbound} ensures
that the equation $\square^\dagger v=\tilde\eta \square^\dagger u$ is solvable in $L^2(\complex^n)$,
and that there exists a solution satisfying 
\begin{equation} \norm{v}_{L^2}\le C\kappa^{-1}\norm{\tilde\eta \square^\dagger u}_{L^2} \le\lambda^{-cN},
\end{equation}
provided that $\lambda\ge \Lambda(N)$.

Now $\square^\dagger(u-v)\equiv 0$ where $\tfrac12<|z|<2$, so Lemma~\ref{lemma:four} can be applied to conclude that
\begin{equation} |(u-v)(z)|\le C e^{-c\kappa} = Ce^{-cr^2\lambda} \ \text{ for $\tfrac34\le|z|\le\tfrac43$.} \end{equation}
Therefore in this same region,
\begin{equation} |\scriptg_\lambda(rz,0)|\le  Ce^{-cr^2\lambda} + C\lambda^{-cN} \end{equation}
for all $\lambda\ge\Lambda(N)$.

Equivalently, by choosing $r = |z|^{-1}$, we find that there exists a constant $B<\infty$ 
such that for all $\lambda\ge\Lambda(N)$ and all $|\zeta|\ge B\lambda^{-1/2}$,
\begin{equation} |\scriptg_\lambda(\zeta,0)|\le  Ce^{-c\lambda|\zeta|^2} + C\lambda^{-cN}
= Ce^{-c\lambda|\zeta|^2} + Ce^{-cN\log\lambda}. \end{equation}

If $A_0$ is sufficiently large, if $A<\infty$ is fixed, and if 
$A_0\lambda^{-1/2}\sqrt{\log\lambda}\le|\zeta|\le A\lambda^{-1/2}\sqrt{\log\lambda}$, 
choose $N = A^2$ to obtain
\begin{equation} |\scriptg_\lambda(\zeta,0)|\le  Ce^{-c\lambda|\zeta|^2}. \end{equation}
After reversing the change of variables made above, this is the desired bound
$|\scriptg_\lambda(z,z')|\le  Ce^{-c\lambda \rho(z,z')^2}$. 

This analysis cannot be extended to a larger range of $|\zeta|$, because bounds only hold for
$\lambda\ge \Lambda(N)$ and a larger range would require that $N$ depend on $|\zeta|$,
hence that $N$ depend on $\lambda$, introducing circularity into the reasoning.

Since $\scriptg_\lambda(z,z')$ is a solution on the complement of the diagonal $z=z'$ 
of homogeneous elliptic linear partial differential equations,
separately with respect to each of the two variables $z,z'$, 
and since the coefficients of these equations are $O(\lambda^2)$ in any $C^M$ norm,
it follows from routine bootstrapping arguments that each derivative
of $\scriptg_\lambda$ satisfies the same upper bound with a possibly smaller value of the constant $c>0$.
Each of the finitely many steps in the bootstrapping process loses at most a factor of $C\lambda^2$. 
Since \[\lambda^C e^{-A\sqrt{\lambda\log\lambda}} \le e^{-(A-1)\sqrt{\lambda\log\lambda}}\]
for all sufficiently large $\lambda$, the loss of finitely many such factors is of no importance here.  \qed



\begin{thebibliography}{99}

\bibitem{bermanBS}
R.~Berman, B.~Berndtsson, J.~Sj\"ostrand,
{\em A direct approach to Bergman kernel asymptotics for positive line bundles}, 
Ark. Mat. 46 (2008), no. 2, 197--217

\bibitem{bob} B.~Berndtsson, 
{\em Bergman kernels related to Hermitian line bundles over compact complex manifolds}, 
in Explorations in complex and Riemannian geometry, 1--17, 
Contemp. Math., 332, Amer. Math. Soc., Providence, RI, 2003

\bibitem{BSZ1}
P.~Bleher, B.~Shiffman,  and S.~Zelditch,  
{\em Universality and scaling of correlations between zeros on complex manifolds}, 
Invent. Math. 142 (2000), no. 2, 351--395

\bibitem{BSZ2}
\bysame,
{\em Poincar\'e-Lelong approach to universality and scaling of correlations between zeros}, 
Comm. Math. Phys. 208 (2000), no. 3, 771--785

\bibitem{boutetsjostrand}
L.~Boutet de Monvel and J.~Sj\"ostrand,  Sur la singularit\'e des noyaux de Bergman et de Szeg\"o, 
in Journ\'ees: \'Equations aux d\'eriv\'ees partielles de Rennes (1975), 
Ast\'erisque 34--35, pp. 123--164, Soc. Math. France, Paris, 1976

\bibitem{catlin}
D.~Catlin, {\em The Bergman kernel and a theorem of Tian}, 
in: Analysis and Geometry in Several Complex Variables, G. Komatsu and M. Kuranishi, eds., Birkh\"auser, Boston 1999

\bibitem{chinni}
G.~Chinni, {\em A proof of hypoellipticity for Kohn's operator via FBI}, 
Rev. Mat. Iberoam. 27 (2011), no. 2, 585--604

\bibitem{christweighted}
M.~Christ, 
{\em On the $\bar\partial$¯ equation in weighted $L^2$ norms in ${\mathbf C}^1$}, 
J. Geom. Anal. 1 (1991), no. 3, 193--230

\bibitem{christcounter}
\bysame, 
{\em Slow off-diagonal decay for Szeg\"o kernels associated to smooth Hermitian line bundles}, 
Harmonic analysis at Mount Holyoke (South Hadley, MA, 2001), 77--89,
Contemp. Math., 320, Amer. Math. Soc., Providence, RI, 2003

\bibitem{christzeldconjecture}
\bysame, 
{\em On a conjecture of Zelditch regarding Bergman kernels}, in preparation

\bibitem{delin}
H.~Delin,
{\em Pointwise estimates for the weighted Bergman projection kernel in ${\mathbf C}^n$, 
using a weighted $L^2$ estimate for the $\bar\partial$ equation}, 
Ann. Inst. Fourier (Grenoble) 48 (1998), no. 4, 967--997

\bibitem{grigissjostrand}
A.~Grigis and J.~Sj\"ostrand,
{\em Front d'onde analytique et sommes de carr\'es de champs de vecteurs},
Duke Math. J. 52 (1985), no. 1, 35--51 


\bibitem{hormander}
L.~H\"ormander, {\em An introduction to complex analysis in several variables},
Third edition. North-Holland Mathematical Library, 7. North-Holland Publishing Co., Amsterdam, 1990

\bibitem{kohn}
J.~J.~Kohn,
{\em The range of the tangential Cauchy-Riemann operator}, Duke Math. J. 53 (1986), no. 2, 525--545

\bibitem{lindholm}
N.~Lindholm,
{\em Sampling in weighted $L^p$ spaces of entire functions in ${\mathbf C}^n$ 
and estimates of the Bergman kernel}, J. Funct. Anal. 182 (2001), no. 2, 390--426

\bibitem{shiffmanzelditch}
B.~Shiffman  and S.~Zelditch,  
{\em Distribution of zeros of random and quantum chaotic sections of positive line bundles}, 
Comm. Math. Phys. 200 (1999), no. 3, 661--683

\bibitem{sjostrand1}
J.~Sj\"ostrand,
{\em Singularit\'es analytiques microlocales},  Ast\'erisque, 95, 1--166, Ast\'erisque, 95, Soc. Math. France, Paris, 1982

\bibitem{sjostrand2}
\bysame,
{\em Analytic wavefront sets and operators with multiple characteristics}, 
Hokkaido Math. J. 12 (1983), no. 3, part 2, 392--433

\bibitem{tartakoff}
D.~Tartakoff,
{\em On the local real analyticity of solutions to $\square_b$
and the $\bar\partial$--Neumann problem}, Acta Math. 145 (1980), 117-204.

\bibitem{tian}
G.~Tian, {\em On a set of polarized K\"ahler metrics on algebraic manifolds}, J. Differential Geom. 32 (1990), 99--130

\bibitem{treves1}
F.~Treves,
{\em Analytic-hypoelliptic partial differential equations of principal type}, 
Comm. Pure Appl. Math. 24 (1971), 537--570

\bibitem{treves2}
\bysame
{\em Analytic hypo-ellipticity of a class of pseudodifferential operators
with double characteristics and applications to the $\bar\partial$--Neumann problem},
Comm. Partial Differential Equations 3 (1978), no. 6-7, 475--642.

\bibitem{zelditch}
S. Zelditch, Szeg\"o kernels and a theorem of Tian, Int. Math. Res. Notices 6 (1998), 317--331.

\end{thebibliography}
\end{document}